\pdfoutput=1 
\DeclareSymbolFont{AMSb}{U}{msb}{m}{n}
\documentclass[11pt,noamsfonts,a4paper]{amsart}
\usepackage[left=1in, right=1in, top=1in, bottom=1in]{geometry}
\usepackage{mathtools}
\usepackage{braket}
\usepackage[charter]{mathdesign}
\usepackage[tracking]{microtype}
\usepackage{tabularx}
\usepackage{etoolbox}
\usepackage{graphicx}
\usepackage{stmaryrd} 

\usepackage{comment}

\usepackage{amsmath,amsthm}
\newtheoremstyle{pineapple}%
  {1em}{1em}%
  {\itshape}{}%
  {\bfseries}{. ---}
  {0.5em}{}

\newtheoremstyle{durian}%
  {1em}{1em}%
  {}{}%
  {\bfseries}{. ---}
  {0.5em}{}

\makeatletter
\def\swappedhead#1#2#3{%
  \thmnumber{\@upn{\the\thm@headfont#2\@ifnotempty{#1}{.~}}}%
  \thmname{#1}%
  \thmnote{ {\the\thm@notefont(#3)}}}
\makeatother

\makeatletter
\newcommand*\rel@kern[1]{\kern#1\dimexpr\macc@kerna}
\newcommand*\widebar[1]{%
  \begingroup
  \def\mathaccent##1##2{%
    \rel@kern{0.8}%
    \overline{\rel@kern{-0.8}\macc@nucleus\rel@kern{0.2}}%
    \rel@kern{-0.2}%
  }%
  \macc@depth\@ne
  \let\math@bgroup\@empty \let\math@egroup\macc@set@skewchar
  \mathsurround\z@ \frozen@everymath{\mathgroup\macc@group\relax}%
  \macc@set@skewchar\relax
  \let\mathaccentV\macc@nested@a
  \macc@nested@a\relax111{#1}%
  \endgroup
}
\makeatother

\makeatletter
\def\@sect#1#2#3#4#5#6[#7]#8{%
  \edef\@toclevel{\ifnum#2=\@m 0\else\number#2\fi}%
  \ifnum #2>\c@secnumdepth \let\@secnumber\@empty
  \else \@xp\let\@xp\@secnumber\csname the#1\endcsname\fi
  \@tempskipa #5\relax
  \ifnum #2>\c@secnumdepth
    \let\@svsec\@empty
  \else
    \refstepcounter{#1}%
    \edef\@secnumpunct{%
      \ifdim\@tempskipa>\z@ 
        \@ifnotempty{#8}{.~}%
      \else
        \@ifempty{#8}{.}{.~}%
      \fi
    }%
    \@ifempty{#8}{%
      \ifnum #2=\tw@ \def\@secnumfont{\bfseries}\fi}{}%
    \protected@edef\@svsec{%
      \ifnum#2<\@m
        \@ifundefined{#1name}{}{%
          \ignorespaces\csname #1name\endcsname\space
        }%
      \fi
      \@seccntformat{#1}%
    }%
  \fi
  \ifdim \@tempskipa>\z@ 
    \begingroup #6\relax
    \@hangfrom{\hskip #3\relax\@svsec}{\interlinepenalty\@M #8\par}%
    \endgroup
    \ifnum#2>\@m \else \@tocwrite{#1}{#8}\fi
  \else
  \def\@svsechd{#6\hskip #3\@svsec
    \@ifnotempty{#8}{\ignorespaces#8\unskip
       \@addpunct.}%
    \ifnum#2>\@m \else \@tocwrite{#1}{#8}\fi
  }%
  \fi
  \global\@nobreaktrue
  \@xsect{#5}}
\makeatother

\makeatletter
\def\@seccntformat#1{%
  \protect\textup{\protect\@secnumfont
    \ifnum\pdfstrcmp{subsection}{#1}=0 \bfseries\fi
    \csname the#1\endcsname
    \protect\@secnumpunct
  }%
}
\makeatother

\theoremstyle{pineapple}

\newtheorem*{IntroTheorem*}{Theorem}

\swapnumbers
\newtheorem{Theorem}[section]{Theorem}
\newtheorem{Lemma}[section]{Lemma}

\theoremstyle{durian}

\newtheorem{Remark}[section]{}

\usepackage[dvipsnames]{xcolor}
\usepackage{hyperref} 
\hypersetup{
  colorlinks = true,
  linkcolor = Fuchsia,
  urlcolor = Fuchsia,
  citecolor = ForestGreen,
  linkbordercolor = {white}}
\usepackage{tikz-cd}
\usetikzlibrary{arrows}

\tikzset{
  symbol/.style={
    draw=none,
    every to/.append style={
      edge node={node [sloped, allow upside down, auto=false]{$#1$}}}
  }
}

\usepackage{enumitem}
\setlist[1]{labelindent=\parindent}
\setlist[1]{labelsep=0.5em}
\setlist[enumerate,1]{label={\upshape (\roman*)}, ref={\upshape (\roman*)}}

\makeatletter
\newcommand{\leqnomode}{\tagsleft@true\let\veqno\@@leqno}
\newcommand{\reqnomode}{\tagsleft@false\let\veqno\@@eqno}
\makeatother

\usepackage[utf8]{inputenc}

\tikzset{>={Straight Barb[length=2pt,width=4pt]}, commutative diagrams/arrow style=tikz}

\makeatletter
\let\c@equation\c@subsection

\makeatother

\DeclareMathOperator{\Fr}{Fr}

\makeatletter
\newcommand*{\coloneqq}{\mathrel{\rlap{%
           \raisebox{0.3ex}{$\m@th\cdot$}}%
           \raisebox{-0.3ex}{$\m@th\cdot$}}%
           =}
\newcommand{\eqqcolon}{=%
           \mathrel{\rlap{%
           \raisebox{0.3ex}{$\m@th\cdot$}}%
           \raisebox{-0.3ex}{$\m@th\cdot$}}}
\makeatother

\newcommand{\parref}[1]{{\bf\ref{#1}}}

\DeclareMathOperator{\rank}{rank}

\DeclareMathOperator{\Hom}{Hom}

\newcommand{\kk}{\mathbf{k}}

\newcommand{\sO}{\mathcal{O}}

\makeatletter
\newcommand{\smallbullet}{} 
\DeclareRobustCommand\smallbullet{%
  \mathord{\mathpalette\smallbullet@{0.75}}%
}
\newcommand{\smallbullet@}[2]{%
  \vcenter{\hbox{\scalebox{#2}{$\m@th#1\bullet$}}}%
}
\makeatother

\newcommand{\PP}{\mathbf{P}}

\title{Free curves in Fano hypersurfaces must have high degree}
 \author{Raymond Cheng}
 \address{Institute of Algebraic Geometry \\
   Leibniz University Hannover \\
   Germany
 }
 \email{cheng@math.uni-hannover.de}

\begin{document}
\begin{abstract}
The purpose of this note is to show that the minimal \(e\) for which every
smooth Fano hypersurface of dimension \(n\) contains a free rational curve of
degree at most \(e\) cannot be bounded by a linear function in \(n\) when
the base field has positive characteristic. This is done by providing a
super-linear bound on the minimal possible degree of a free curve in certain
Fermat hypersurfaces.
\end{abstract}
\maketitle

\thispagestyle{empty}
\section*{Introduction}
The geometry of smooth projective Fano varieties is controlled by the rational
curves they contain. Seminal work \cite{KMM, Campana} of
Koll\'ar--Miyaoka--Mori and Campana show that, over a field of characteristic
\(0\), every smooth projective Fano variety \(X\) contains a rational curve
\(\varphi \colon \PP^1 \to X\) that, informally speaking, can be deformed to
pass through \(r+1\) general points of \(X\) for any chosen \(r \geq 0\): in
other words, \(X\) is \emph{separably rationally connected}. The precise
condition on the curve \(\varphi\) is that
\(\mathrm{H}^1(\PP^1,\varphi^*\mathcal{T}_X \otimes \sO_{\PP^1}(-r-1)) = 0\);
by way of terminology, \(\varphi\) is said to be \emph{free} or \emph{very
free} when \(r = 0\) or \(r = 1\), respectively. See \cite{Kollar,
Debarre:HDAG} for a presentation of this theory.

Whether smooth projective Fano varieties over a field of positive
characteristic are separably rationally connected is a long-standing open
question. Results are fragmentary even for smooth Fano hypersurfaces in
projective space: The general Fano hypersurface is separably rationally
connected by \cite{Zhu, CZ14, Tian, CR17}; notably, the work of Tian reduces
the problem of separable rational connectedness to \emph{separable
uniruledness}---that is, the existence of a free rational curve---a problem
that often is simpler because free curves typically have significantly lower
degree than very free curves. More recently, \cite[Theorems
3.10 and 3.24]{STZ} and \cite[Corollary 9]{ST} shows that all smooth Fano
hypersurfaces with degree less than the characteristic are separably rationally
connected and even that, up to a minor condition, such hypersurfaces always
contain either free lines or conics. See also \cite[Theorem 34]{LP},
\cite[Theorem 1.5]{BS}, and \cite[Theorem 1.12]{BLLRS} for related results.

The main result of this note is that, nevertheless and contrary to experience,
the minimal \(e\) such that there exists a free rational curve of degree
\(\leq e\) on every smooth Fano hypersurface cannot be bounded by a linear
function in the dimension (or degree); contrast this with the fact that every
smooth Fano hypersurface in characteristic \(0\) contains either a free line or
conic.

\begin{IntroTheorem*}
For any algebraically closed field \(\kk\) of characteristic \(p > 0\),
\[
\limsup_{n \to \infty} \frac{1}{n}
\inf\Set{ e \in \mathbf{Z} |
\begin{array}{c}
\text{for every smooth Fano hypersurface of dimension \(n\) over \(\kk\)} \\
\text{there exists a free rational curve of degree \(\leq e\)}
\end{array}}
=
\infty.
\]
\end{IntroTheorem*}

It suffices to describe one sequence of increasingly high-dimensional
hypersurfaces without low degree free curves, and this is given by the Fermat
hypersurface \(X\) of degree \(q+1\) in \(\PP^{q+1}\) with
\(q \coloneqq p^\nu\) for \(\nu \geq 1\). Shen has studied very free curves in
\(X\) and showed in \cite{Shen} that many integers below \(q^2\) cannot be the
degree of such a curve; a reformulation of this for free curves is given in
\parref{r-at-most-m}. The main technical result \parref{free-bound} excludes a
complementary set of degrees, and they are combined in
\parref{superlinear-bound} to show that there are no gaps below \(q^{3/2}-q\).
The method is to exploit a certain tension arising from the curious
differential geometry of these hypersurfaces: the structure of the equation
implies that, on the one hand, free curves span either the ambient projective
space or else a hyperplane and, on the other hand, there are unexpected
constraints on the coordinate functions of the curve.

These hypersurfaces are well-known to be exceptional and exemplify many
positive characteristic phenomena: see \cite[pp.7--11]{thesis} for a general
survey regarding these hypersurfaces. What is fascinating is that these
hypersurfaces contain many, many rational curves---they are unirational!---and
so the challenge is to develop techniques to study their spaces of rational
curves: see \cite{fano-schemes, qbic-threefolds} for work in this direction.

\smallskip
\noindent\textbf{Acknowledgements. ---}
This note originates from a question posed to me long ago by Aise Johan
de Jong; much gratitude for the many discussions and interest over the years.
Thanks to Jason Starr and Remy van Dobben de Bruyn with whom I shared helpful
conversations on this topic, and the anonymous referees for their careful
reading and helpful comments. I was supported by a Humboldt Research Fellowship
during the preparation of this note.

\section*{Free curves in the Fermat hypersurface}
In what follows, let
\(X \coloneqq \mathrm{V}(T_0^{q+1} + \cdots + T_{q+1}^{q+1})\) be the Fermat
hypersurface of degree \(q+1\) in \(\PP^{q+1}\). The \emph{embedded tangent
bundle} \(\mathcal{E}_X\) is the vector bundle on \(X\) whose fibre at a point
\(x\) is the linear space underlying the embedded tangent space of \(X\) at
\(x\); it fits into a short exact sequence
\[
0 \to \sO_X \to \mathcal{E}_X \to \mathcal{T}_X \to 0.
\]
The extension class is the pullback via the tangent map
\(\mathcal{T}_X \to \mathcal{T}_{\PP^{q+1}}\rvert_X\) of the class of the Euler
sequence, and so there is another short exact sequence
\[
0 \to \mathcal{E}_X \to \sO_X(1)^{\oplus q+2} \to \mathcal{N}_{X/\PP^{q+1}} \to 0
\]
where the second map is
\((\psi_0,\ldots,\psi_{q+1}) \mapsto \sum\nolimits_{i=0}^{q+1} T_i^q \cdot \psi_i\).
Remarkably, as already observed by \cite[Equation (2)]{Shen}, this sequence
twisted down by \(\sO_X(-1)\) is isomorphic to the pullback of the dual Euler
sequence by the \(q\)-power Frobenius morphism \(\Fr\). In particular, there is
an isomorphism
\[
\mathcal{E}_X(-1) \cong \Fr^*(\Omega^1_{\PP^{q+1}}(1)\rvert_X).
\]

Let \(\varphi \colon \PP^1 \to X\) be a nonconstant morphism of degree
\(e = mq + r\), where \(m,r \in \mathbf{Z}\) and \(0 \leq r \leq q-1\); to
simplify notation, \(\varphi\) will sometimes be viewed as a morphism into
\(\PP^{q+1}\). Viewing \(\mathcal{E}_X\) as an extension of \(\mathcal{T}_X\),
it follows that \(\varphi\) is free if and only if
\(\mathrm{H}^1(\PP^1, \varphi^*\mathcal{E}_X \otimes \sO_{\PP^1}(-1)) = 0\).
Combined with the isomorphism above and the projection formula, this implies
that, if \(\varphi\) is free,
\begin{align*}
0
= \mathrm{H}^1(\PP^1, \varphi^*\mathcal{E}_X \otimes \sO_{\PP^1}(-1))
& = \mathrm{H}^1(\PP^1, \varphi^*\Fr^*(\Omega^1_{\PP^{q+1}}(1)\rvert_X) \otimes \sO_{\PP^1}(e-1)) \\
& = \mathrm{H}^1(\PP^1, \varphi^*(\Omega^1_{\PP^{q+1}}(1)) \otimes \Fr_*\sO_{\PP^1}(e-1)).
\end{align*}
Since
\(\Fr_*\sO_{\PP^1}(e-1) \cong \sO_{\PP^1}(m)^{\oplus r} \oplus \sO_{\PP^1}(m-1)^{\oplus q-r}\),
this shows that:

\begin{Lemma}\label{vanishing-H1}
If \(\varphi \colon \PP^1 \to X\) is free, then
\(\mathrm{H}^1(\PP^1, \varphi^*\Omega^1_{\PP^{q+1}} \otimes \sO_{\PP^1}(e + m - 1)) = 0\).
\qed
\end{Lemma}

The converse also holds. Since
\(\chi(\PP^1, \varphi^*\Omega^1_{\PP^{q+1}} \otimes \sO_{\PP^1}(e+m-1)) = (q+2)m - e - m = m - r\),
this gives a relation between \(m\) and \(r\) when \(\varphi\) is free,
recovering \cite[Theorem 1.7]{Shen}:

\begin{Lemma}\label{r-at-most-m}
If \(\varphi \colon \PP^1 \to X\) is free, then \(r \leq m\). \qed
\end{Lemma}

The following gives a geometric restriction on free curves in \(X\), and stands
in stark contrast to the fact, see \cite[V.4.4]{Kollar}, that a general
smooth Fano hypersurface contains either a free line or conic:

\begin{Lemma}\label{free-nondegen}
If \(\varphi \colon \PP^1 \to X\) is free, then \(\varphi(\PP^1)\) either
spans \(\PP^{q+1}\) or a hyperplane. Moreover, in the latter case, \(e = mq\)
for positive \(m\), and
\(\varphi^*(\Omega^1_{\PP^{q+1}}(1)) \cong \sO_{\PP^1}(-m)^{\oplus q} \oplus \sO_{\PP^1}\).
\end{Lemma}

\begin{proof}
Identifying \(\mathcal{E}_X = \ker(\sO_X(1)^{\oplus q+2} \to \sO_X(q+1))\) and
using that \(\varphi\) is free shows that
\[
\dim\mathrm{H}^0(\PP^1, \varphi^*\mathcal{E}_X \otimes \sO_{\PP^1}(-1)) = e.
\]
If \(\varphi\) were contained in a hyperplane \(\PP^q \subset \PP^{q+1}\), then
juxtaposing the two Euler sequences shows that
\(\varphi^*(\Omega^1_{\PP^{q+1}}(1))\) would split as
\(\varphi^*(\Omega^1_{\PP^q}(1)) \oplus \sO_{\PP^1}\). Combined with the fact
that \(\mathcal{E}_X(-1) \cong \Fr^*(\Omega^1_{\PP^{q+1}}(1)\rvert_X)\), this
would imply that
\[
e = \dim\mathrm{H}^0(\PP^1, \varphi^*\mathcal{E}_X \otimes \sO_{\PP^1}(-1))
= \dim\mathrm{H}^0(\PP^1, \Fr^*\varphi^*(\Omega^1_{\PP^q}(1)) \otimes \sO_{\PP^1}(e-1)) + e.
\]
Thus \(\Fr^*\varphi^*(\Omega^1_{\PP^q}(1)) \otimes \sO_{\PP^1}(e-1)\)
cannot have global sections. Freeness of \(\varphi\) also implies that it has
no higher cohomology, and so it must be isomorphic to
\(\sO_{\PP^1}(-1)^{\oplus q}\). However, since
\(\varphi^*(\Omega^1_{\PP^q}(1))\) is a vector bundle of degree \(-e\), there
are integers \(a_i \geq 0\) with \(a_1 + \cdots + a_q = e\) such that
\[
\Fr^*\varphi^*(\Omega^1_{\PP^q}(1)) \otimes \sO_{\PP^1}(e-1) \cong
\bigoplus\nolimits_{i = 1}^q \sO_{\PP^1}(-a_iq+e-1).
\]
Therefore \(e = a_iq\) for each \(i\), and so
\(a_1 = \cdots = a_q = m\), \(e = mq\), and
\(\varphi^*(\Omega^1_{\PP^q}(1)) \cong \sO_{\PP^1}(-m)^{\oplus q}\).
Since \(\varphi\) is nonconstant, \(m\) is positive and \(\varphi(\PP^1)\)
spans the hyperplane \(\PP^q\).
\end{proof}

\begin{Remark}\label{frobenius-decomposition}
Viewing \(\varphi\) as a morphism into \(\PP^{q+1}\) and letting
\((\varphi_0: \cdots :\varphi_{q+1})\) be its components,
\parref{free-nondegen} means that the \(\varphi_i\) enjoy at most \(1\) linear
relation in \(\mathrm{H}^0(\PP^1,\sO_{\PP^1}(e))\). Already, this implies
\(e \geq q\). That \(\varphi\) factors through \(X\) means that
\(\sum_{i = 0}^{q+1} \varphi_i^q \cdot \varphi_i = 0\). Upon choosing
homogeneous coordinates \((S_0:S_1)\) for \(\PP^1\), there is a unique
decomposition
\[
\varphi_i =
\sum\nolimits_{j = 0}^r \zeta_{ij}^q \cdot S_0^j S_1^{r-j} +
\sum\nolimits_{k = 1}^{q-r-1} \eta_{ik}^q \cdot S_0^{r+k} S_1^{q-k}
\]
where \(\zeta_{ij} \in \mathrm{H}^0(\PP^1, \sO_{\PP^1}(m))\) and
\(\eta_{ik} \in \mathrm{H}^0(\PP^1, \sO_{\PP^1}(m-1))\). Note that this
decomposition provides a specific choice of isomorphism
\(\Fr_*\sO_{\PP^1}(e) \cong \sO_{\PP^1}(m)^{\oplus r+1} \oplus \sO_{\PP^1}(m-1)^{\oplus q-r-1}\),
and which shall be used below. Substituting this into
the equation of \(X\) and using the fact that \(q\) is a power of the ground
field characteristic shows that
\[
0 =
\sum\nolimits_{j = 0}^r \Big(\sum\nolimits_{i = 0}^{q+1} \varphi_i \zeta_{ij}\Big)^q \cdot S_0^j S_1^{r-j} +
\sum\nolimits_{k = 1}^{q-r-1} \Big(\sum\nolimits_{i = 0}^{q+1} \varphi_i \eta_{ik}\Big)^q \cdot S_0^{r+k} S_1^{q-k}
\]
and so, upon looking at exponents modulo \(q\),
\(\sum_{i = 0}^{q+1} \varphi_i\zeta_{ij} = \sum_{i = 0}^{q+1} \varphi_i \eta_{ik} = 0\) for all \(j\) and \(k\). These
relations may impose further linear relations on the \(\varphi_i\), so give
strong restrictions on the degree \(e\):
\end{Remark}

\begin{Theorem}\label{free-bound}
If \(\varphi \colon \PP^1 \to X\) is free, then
\(q+1 \leq m^2 + m + r\) if \(r > 0\) and
\(q \leq m^2 + m\) if \(r = 0\).
\end{Theorem}

\begin{proof}
Consider the linear map
\(
\Phi \colon
\mathrm{H}^0(\PP^{q+1}, \sO_{\PP^{q+1}}(1)) \to
\mathrm{H}^0(\PP^1,\sO_{\PP^1}(e))
\) defining \(\varphi \colon \PP^1 \to \PP^{q+1}\), so that
the \(i\)-th coordinate \(T_i\) maps to \(\varphi_i\). Identify the
target of \(\Phi\) as
\[
\mathrm{H}^0(\PP^1, \sO_{\PP^1}(e)) \cong
\mathrm{H}^0(\PP^1, \Fr_*\sO_{\PP^1}(e)) \cong
\mathrm{H}^0(\PP^1, \sO_{\PP^1}(m))^{\oplus r+1} \oplus
\mathrm{H}^0(\PP^1, \sO_{\PP^1}(m-1))^{\oplus q-r-1}
\]
where \(\Fr \colon \PP^1 \to \PP^1\) is the \(q\)-power Frobenius morphism.
Let \(\Phi_1\) and \(\Phi_2\) be the linear maps obtained by post-composing
\(\Phi\) with projection to \(\mathrm{H}^0(\PP^1, \sO_{\PP^1}(m))^{\oplus r+1}\)
and \(\mathrm{H}^0(\PP^1, \sO_{\PP^1}(m-1))^{\oplus q-r-1}\), respectively.
Elementary linear algebra gives
\[
\rank\Phi \leq \rank \Phi_1 + \rank \Phi_2 \leq (r+1)(m+1) + \rank\Phi_2.
\]

To bound the rank of \(\Phi_2\), let
\(\Phi_{2,k} \colon \mathrm{H}^0(\PP^{q+1},\sO_{\PP^{q+1}}(1)) \to \mathrm{H}^0(\PP^1,\sO_{\PP^1}(m-1))\)
be the further projection to the \(k\)-th component of its
target, so that \(\Phi_{2,k}(T_i) = \eta_{ik}\), notation as in
\parref{frobenius-decomposition}. The discussion of
\parref{frobenius-decomposition} means that the \(\Phi_{2,k}\) lie in the
kernel of the linear map
\begin{align*}
\Hom(\mathrm{H}^0(\PP^{q+1},\sO_{\PP^{q+1}}(1)), \mathrm{H}^0(\PP^1, \sO_{\PP^1}(m-1)))
& \longrightarrow
\mathrm{H}^0(\PP^1, \sO_{\PP^1}(e + m - 1)) \\
\Psi & \longmapsto \sum\nolimits_{i = 0}^{q+1} \varphi_i \cdot \Psi(T_i).
\end{align*}
The kernel of this map is isomorphic to
\(\mathrm{H}^0(\PP^1, \varphi^*\Omega^1_{\PP^{q+1}} \otimes \sO_{\PP^1}(e+m-1))\):
Indeed, pull back the Euler sequence via \(\varphi\) and twist up by
\(\sO_{\PP^1}(e+m-1)\) to obtain the exact sequence
\[
0 \to
\varphi^*\Omega_{\PP^{q+1}}^1 \otimes \sO_{\PP^1}(e+m-1) \to
\sO_{\PP^1}(m-1) \otimes \mathrm{H}^0(\PP^{q+1}, \sO_{\PP^{q+1}}(1)) \to
\sO_{\PP^1}(e+m-1) \to
0,
\]
then use the choice of coordinates \(T_i\) to make the identification
\[
\mathrm{H}^0(\PP^1, \sO_{\PP^1}(m-1)) \otimes \mathrm{H}^0(\PP^{q+1}, \sO_{\PP^{q+1}}(1))
\cong
\Hom(\mathrm{H}^0(\PP^{q+1},\sO_{\PP^{q+1}}(1)), \mathrm{H}^0(\PP^1, \sO_{\PP^1}(m-1))).
\]

On the one hand, \parref{vanishing-H1} and the Euler characteristic computation
following it shows that this kernel has dimension \(m-r\), and so at most
\(m - r\) of the \(q-r-1\)
components \(\Phi_{2,k}\) of \(\Phi_2\) are linearly
independent. Since \(\rank\Phi_{2,k} \leq m\) for each \(k\),
\(\rank\Phi_2 \leq (m-r)m\). Therefore
\[
\rank\Phi \leq (r+1)(m+1) + (m-r)m = m^2 + m + r + 1.
\]
On the other hand, \parref{free-nondegen} means that \(\Phi\) has rank \(q+2\)
if \(r > 0\) and rank at least \(q+1\) if \(r = 0\). Put together, these give
the inequalities of the statement.
\end{proof}

Combining \parref{r-at-most-m} with \parref{free-bound} and grossly
underestimating shows that \(\sqrt{q} - 1 \leq m\). Writing \(e = mq + r\) then
yields a super-linear bound on the minimal degree of a free curve in \(X\):

\begin{Theorem}\label{superlinear-bound}
If \(\varphi \colon \PP^1 \to X\) is free, then \(e \geq q^{3/2} - q\).
\qed
\end{Theorem}

The strongest restrictions on possible degrees of free curves are provided by
\parref{r-at-most-m} and \parref{free-bound}. They give the following for the first
few prime powers:
\smallskip
\begin{center}
\begin{tabular}{l|c|c|c|c|c|c|c|c|c|c|c|c|c|c|c|c|c|c}
\(q\) & \(2\) & \(3\) & \(4\) & \(5\)  & \(7\)  & \(8\)  & \(9\)  & \(11\) & \(13\) & \(16\) & \(17\) & \(19\) & \(23\) & \(25\) & \(27\) & \(29\) & \(31\) & \(32\) \\
\hline
\(e_{\mathrm{min}}\) & \(3\) & \(6\) & \(8\) & \(10\) & \(16\) & \(24\) & \(27\) & \(33\) & \(41\) & \(64\) & \(68\) & \(76\) & \(96\) & \(125\) & \(135\) & \(145\) & \(157\) & \(163\)
\end{tabular}
\end{center}
\smallskip
Free curves achieving these lower bounds are known to exist in the first few
cases: see \cite[p.6]{Madore}, \cite[p.69]{Conduche}, and \cite{BDENY} for
\((q,e) = (2,3), (3,6), (4,8)\), respectively. As far as I know, no free
curves are known to exist when \(q \geq 5\). Finally, observe that the
arguments of \parref{vanishing-H1} further imply that if \(\varphi \colon \PP^1
\to X\) is free, then it is \(r\)-free, in the sense that
\(\mathrm{H}^1(\PP^1, \varphi^*\mathcal{T}_X \otimes \sO_{\PP^1}(-r-1)) =
0\), so that for many \(q\) above, a minimal possible free curve is
automatically very free.

\bibliographystyle{amsalpha}
\bibliography{main}
\end{document}